\documentclass[11pt]{amsart}
\usepackage{amsmath}
\usepackage{amssymb, latexsym, amsfonts, enumitem, tikz, fancyhdr, bm, mathrsfs}
\usepackage{todonotes}
\usepackage{xcolor}
\usepackage[normalem]{ulem} 
\usepackage{fullpage}

\newtheorem{thm}{Theorem}[section]
\newtheorem{prop}[thm]{Proposition}
\newtheorem{lem}[thm]{Lemma}
\newtheorem{defn}[thm]{Definition}

\newtheorem{rem}[thm]{Remark}
\newtheorem{cor}[thm]{Corollary}

\newtheorem*{theorem*}{Claim}

\usepackage{comment}

\usepackage{enumitem}

\begin{document}

\title[Unique opeartor space structure]{A two-dimensional subspace of $\ell_\infty^3$ has a unique operator space structure if and only if it is isometric to $\ell_\infty^2$}
\author[Gupta]{Rajeev Gupta}
\author[Mal]{Arpita Mal}
\author[Misra]{Gadadhar Misra}
\author[Ray]{Samya Kumar Ray}

\address[Gupta]{School of Mathematics and Computer Science, Indian Institute of Technology Goa, Goa - 403401}
\email[Gupta]{\text rajeev@iitgoa.ac.in}

\address[Mal]{Dhirubhai Ambani University, Gandhinagar, India}
\email[Mal]{\text arpita\_mal@dau.ac.in}

\address[Misra]{Indian Statistical Institute, Bangalore, India and Indian Institute of Technology, Gandhinagar, India}
\email[Misra]{\text gm@isibang.ac.in}

\address[Ray]{The Institute of Mathematical Sciences, 4th Cross Street, CIT Campus, Tharamani, Chennai, Tamil Nadu 600113, India, and Homi Bhabha National Institute, Training School Complex, Anushakti Nagar, Mumbai 400094, India} 
\email[Ray]{\text samya@imsc.res.in}
\begin{abstract}
We show that a two-dimensional complex subspace of $\ell_\infty^3$ carries a
unique operator space structure if and only if it is isometric to
$\ell_\infty^2$; the ``if'' half is the And\^o's theorem.  For each of the remaining
spaces we exhibit an explicit non-normal $Y \in \mathcal{M}_2$, with $(I,Y)$ an extreme
point of the associated matrix ball, for which the induced linear map is
contractive but not completely contractive.  Since every two-dimensional
subspace of $\ell_\infty^3$ has the $2$-summing property, the $2$-summing
property does not imply uniqueness of the operator space structure for two dimensional complex Banach spaces.
\end{abstract}
\maketitle

\section{Introduction}
This paper is a companion to \cite{GMMR}, where we studied the $2$-summing property — equivalently, Property P — for the normed linear spaces $X_A=(\mathbb C^2,\|\cdot\|_A)$, $\|(z_1,z_2)\|_A:=\|z_1A_1+z_2A_2\|$, attached to a pair $A=(A_1,A_2)$ of $d\times d$ complex matrices. For $d=2$, $X_A$ has the $2$-summing property if and only if $X_A$ is isometrically isomorphic to $\ell_\infty^2$ \cite{MPV}. For $d=3$ it persists in the diagonal case: Proposition 4.4 of \cite{AFJS95} shows that  $X_A$ has the $2$-summing property for diagonal $A_1,A_2\in\mathcal M_3$, and in \cite{GMMR} we showed that such an $X_A$ has Property Q introduced in \cite{BM}, which implies Property P, also introduced in \cite{BM}. What remained open there was the value of $\alpha(X_A)$.

In this article, we answer the above question. More specifically, we show that $\alpha(X_A)>1$ for every pair of diagonal $A_1,A_2\in\mathcal M_3$ unless $X_A$ is isometric to $\ell_\infty^2$.
Since $\alpha(X)=\|\mathrm{id}:\mathrm{MIN}(X)\to\mathrm{MAX}(X)\|_{cb}$, this says that $\ell_\infty^2$ is, up to isometry, the only two-dimensional subspace of $\ell_\infty^3$ carrying a unique operator space structure. In particular, already for two-dimensional spaces, the $2$-summing property does not necessarily imply $\alpha=1$. In dimension three this was known: $\ell_\infty^3$ itself has the $2$-summing property \cite{AFJS95, GMMR}, see also \cite{BM}, while Parrott's example gives $\alpha(\ell_\infty^3)>1$. What is new here is that the implication fails in the smallest possible dimension, and that for two-dimensional subspaces of $\ell_\infty^3,$ we determine exactly when it fails.

Every two-dimensional subspace of $\ell_\infty^3$ that is not isometric to $\ell_\infty^2$ is isometric to $(\mathbb C^2,\|\cdot\|_{r,s})$ for some $0<r,s\leq1$ with $r+s>1$, where $\|\cdot\|_{r,s}$ is the norm induced by the diagonal matrices $\operatorname{diag}(1,0,r)$ and $\operatorname{diag}(0,1,s)$, that is,
\[\|(z_1,z_2)\|_{r,s}=\max\bigl\{|z_1|,\ |z_2|,\ |rz_1+sz_2|\bigr\}.\]
It therefore suffices to treat this two-parameter family. 
For each such pair $(r,s),$ we produce a $2\times2$ matrix $Y$ for which the linear map induced by $(I,Y)$ on the dual space $(\mathbb C^2,\|\cdot\|_{r,s})^*$ is contractive but not completely contractive.

Our route to these examples is through homomorphisms. Let $\mathbb B$ be the open unit ball of a norm on $\mathbb C^m$. It is well known that the linear map $L_V$ determined on the dual space by an $m$-tuple $V=(V_1,\ldots,V_m)$ of operators is contractive, respectively completely contractive, if and only if the Parrott-like homomorphism $\varrho_V$ determined by the same tuple on the polynomial algebra $\mathbb C[z_1,\ldots,z_m]$, equipped with the supremum norm on $\mathbb B$, is contractive, respectively completely contractive; see \cite[Lemma 5.1]{vip}. 
{ Recall that 
\[\varrho_V(f) := \begin{pmatrix} f(w) & Df(w) V \\ 0 & f(w)\end{pmatrix},\,\,w\in \Omega,\,\, f\in H^\infty(\Omega),\]
where $Df(w) V = \frac{\partial f}{\partial z_1}(w) V_1 + \cdots  + \frac{\partial f}{\partial z_m}(w) V_m$.

The von Neumann inequality stating that $\|p(T)\| \leq \|p\|_{\infty, \mathbb {D}}$ if and only if $T$ is a contraction  
follows from the Sz.–Nagy unitary (power) dilation theorem. In two variables, Ando's dilation 
theorem for a pair of commuting contractions provides a generalisation of the von Neumann's 
inequality. For any $n \geq 3$, two questions arise naturally. First, whether $\|p(T_1, \ldots, T_n)\| \leq \|p\|_{\infty, \mathbb{D}^n}$
if and only if $(T_1, \ldots, T_n)$ are a commuting $n$-tuple of contractions. Second, if $(T_1, \ldots, T_n)$ is an $n$-tuple of 
commuting contractions on a Hilbert space $\mathcal{H}$, then whether there exists a commuting tuple of 
unitaries $(U_1, \ldots , U_n)$ acting on a Hilbert space $\mathcal{K} \supset \mathcal{H}$ such that  
\[p(T_1, \ldots , T_n) = P_{\mathcal{H}} p(U_1, \ldots, U_n)_{| \mathcal{H}},\]
where $P_{\mathcal{H}}$ is the orthogonal projection of $\mathcal{K}$ on $\mathcal{H}$.
The well-known example due to Parrott settled this second question in the negative for $n=3$. However, 
the commuting tuple of contractions constructed by Parrott satisfies von Neumann's inequality, so the first question remained unanswered. It was 
Varopoulos and Kaijser who later found three commuting contractions that failed von Neumann's inequality. 
}

We choose to work with $\varrho_V$. Using Arveson's boundary dilation theorem, we show that if $(I,Y)$ is an extreme point of $\mathscr V_{r,s}(2)$ — the set of pairs of $2\times2$ matrices inducing a contractive linear map on $(\mathbb C^2,\|\cdot\|_{r,s})^*$ — then the homomorphism it induces is completely contractive if and only if $Y$ is normal. Here $I$ denotes the identity matrix. It remains to produce such a non-normal $Y$, and we do so for every pair $0<r,s\leq1$ with $r+s>1$. Hence for every two-dimensional subspace of $\ell_\infty^3$ not isometric to $\ell_\infty^2,$ there is a contractive homomorphism of $\mathbb C[z_1,z_2]$, in the supremum norm on the corresponding ball, that is not completely contractive. This, together with And\^o's theorem, proves the assertion of the title.

The same characterization can be obtained from Paulsen's factorization theorem \cite[Proposition 2.2]{VernPaulsen}. We have used the dilation theorem because it keeps the argument inside the category of homomorphisms, where the question originates, and because it produces the commuting normal pair explicitly. The homomorphisms we construct are of Parrott's type, induced by a nilpotent pair, and by the equivalence above this costs nothing as far as the existence of a contractive linear map that is not completely contractive is concerned. Determining which contractive homomorphisms of $\mathbb C[z_1,z_2]$ into $\mathcal M_n$, equipped with the supremum norm on $B_{X_A}$, the open unit ball of $X_A$, fail to be completely contractive, without the restriction to this special form, seems to us to be a harder question and remains open.

We conclude by explaining why a dilation theorem is needed at all. Complete contractivity of $L_{(I,Y)}$ asks that $\|P_1\otimes I+P_2\otimes Y\|\leq 1$ for every admissible pair $(P_1,P_2)$ of matrices. The pair $(A_1,A_2)$ is admissible by the very definition of $\|\cdot\|_A$. It is, however, of no use: $A_1$ and $A_2$ are diagonal, so
\[A_1\otimes I+A_2\otimes Y=\operatorname{diag}\bigl(I,\;Y,\;rI+sY\bigr),\]
and the resulting test, $\max\{1,\|Y\|,\|rI+sY\|\}\leq1$, is precisely the hypothesis that $L_{(I,Y)}$ be contractive. The same degeneracy occurs for every simultaneously unitarily diagonalizable pair, see Remark~\ref{rem:notequal}. Non-diagonal test tuples are therefore unavoidable, and Arveson's dilation theorem is what produces them.

Section~2 contains the preliminaries. Section~3 studies the extreme points needed in the proof. Section~4 gives a criterion for complete contractivity using dilation. Section~5 gives the counterexamples and proves the main result.

The fourth-named author has shown in \cite{RayUnique} that there is an uncountable family of mutually non-isometric two-dimensional complex Banach spaces, each with a unique operator space structure.

\section{Preliminaries}\label{sec:prelim}
Fix an $m$-tuple $A:=(A_1, \ldots, A_m)$ of $d\times d$ linearly independent matrices. Define a norm on $\mathbb{C}^m$ by
\begin{equation}
\|z\|_A := \|z_1 A_1 + \cdots + z_m A_m \|_{\ell^d_2 \to \ell^d_2}.
\end{equation}
If $m=d^2$, then $\|\cdot\|_A$ is isometric to the operator norm in $\mathbb{C}^{d\times d}$. Assume $m< d^2$, so that $(\mathbb{C}^m, \|\cdot\|_A)$ is a subspace $X_A$ of $(\mathbb{C}^{d\times d}, \|\cdot\|_{\ell_2^d \to \ell_2^d})$ via the isometric inclusion $\iota(z_1, \ldots , z_m) = z_1 A_1 + \cdots + z_m A_m$.

If $\ell:X_A \to \mathbb{C}$ is a linear functional, then it has an extension $\hat{\ell}$ to $(\mathbb{C}^{d\times d}, \|\cdot\|_{\ell_2^d \to \ell_2^d})$ with $\|\ell\| = \|\hat{\ell}\|$. Then
\[
\ell(\iota (z_1, \ldots , z_m)) = z_1 \operatorname{tr} (A_1C^*) + \cdots + z_m \operatorname{tr} (A_mC^*)
\]
for some $C\in \mathbb{C}^{d\times d}$. Therefore, the unit ball in the dual space $(\mathbb{C}^m, \|\cdot\|_A)^*$ is
\[
\Big\{\Big( \sum_{i,j=1}^d A_1(i,j) \overline{C(i,j)}, \ldots , \sum_{i,j=1}^d A_m(i,j) \overline{C(i,j)}\Big): \|C\|_{\operatorname{tr}} \leq 1\Big\}.
\]

\subsection{The dual unit ball and contractive maps}\label{sec:dual}
\begin{lem}\label{lem:dual-ball-and-contraction}
Let $X := (\mathbb{C}^2, \|\cdot\|_A)$, where
$\|(z_1, z_2)\|_A = \max\bigl\{|z_1|,\, |z_2|,\, |a_1 z_1 + a_2 z_2|\bigr\}$.
\begin{enumerate}
  \item[(i)] The dual norm on $X^*$ is given by
  \[
    \|(\hat\alpha_1, \hat\alpha_2)\|_{X^*}
    \;=\; \min\Bigl\{|\alpha_1| + |\alpha_2| + |\alpha_3|
      \,:\, \hat\alpha_1 = \alpha_1 + a_1 \alpha_3,\; \hat\alpha_2 = \alpha_2 + a_2 \alpha_3\Bigr\},
  \]
  the minimum being attained, and consequently
  \[
    B_{X^*}
    \;=\; \bigl\{(\alpha_1 + a_1 \alpha_3,\; \alpha_2 + a_2 \alpha_3)
      \,:\, |\alpha_1| + |\alpha_2| + |\alpha_3| \leq 1\bigr\}.
  \]
  \item[(ii)] Let $\mathcal{E}$ be any complex normed linear space. For $V_1, V_2 \in \mathcal{E},$ the linear map
  $V = (V_1, V_2) : X^* \to \mathcal{E}$ defined by
  $V(\hat\alpha_1, \hat\alpha_2) := \hat\alpha_1 V_1 + \hat\alpha_2 V_2$
  satisfies
  \[
    \|V\|_{X^* \to \mathcal{E}} \leq 1
    \;\iff\;
    \max\bigl\{\|V_1\|,\; \|V_2\|,\; \|a_1 V_1 + a_2 V_2\|\bigr\} \leq 1.
  \]
\end{enumerate}
\end{lem}
 
\begin{proof}
The norm on $X$ is the $\ell^\infty$ norm pulled back along the isometric embedding
\[
  \iota : X \hookrightarrow \ell^\infty(3),
  \qquad (z_1, z_2) \longmapsto (z_1,\, z_2,\, a_1 z_1 + a_2 z_2).
\]
Every contractive linear functional $\ell : \ell^\infty(3) \to \mathbb{C}$ has the form
\[
  \ell(z_1, z_2, z_3) = \alpha_1 z_1 + \alpha_2 z_2 + \alpha_3 z_3,
  \qquad |\alpha_1| + |\alpha_2| + |\alpha_3| \leq 1,
\]
so the restriction $\hat\ell := \ell \circ \iota : X \to \mathbb{C}$ is given by
\[
  \hat\ell(z_1, z_2) = \hat\alpha_1 z_1 + \hat\alpha_2 z_2,
  \qquad \hat\alpha_i := \alpha_i + a_i \alpha_3 \;\; (i = 1, 2),
\]
and is plainly a contraction. Conversely, any contractive $\hat\ell : X \to \mathbb{C}$
admits a norm-preserving extension to $\ell^\infty(3)$ by the Hahn--Banach theorem.
Therefore $B_{X^*}$ coincides with the image of
$B_{\ell^\infty(3)^*} = B_{\ell^1(3)}$ under
$(\alpha_1, \alpha_2, \alpha_3) \mapsto (\alpha_1 + a_1 \alpha_3,\, \alpha_2 + a_2 \alpha_3)$,
which establishes both formulas in~(i); the infimum is attained because
$B_{\ell^1(3)}$ is compact.
 
For~(ii), using (i),
\[
  V(\hat\alpha_1, \hat\alpha_2)
  \;=\; \hat\alpha_1 V_1 + \hat\alpha_2 V_2
  \;=\; \alpha_1 V_1 + \alpha_2 V_2 + \alpha_3 \bigl(a_1 V_1 + a_2 V_2\bigr),
\]
and by (i) the quotient map $\ell^1(3) \to X^*$ carries $B_{\ell^1(3)}$ onto
$B_{X^*}$. Hence $\|V\|_{X^* \to \mathcal{E}}$ equals the norm of the map
$\ell^1(3) \to \mathcal{E}$ sending the standard basis vectors to
$V_1,\, V_2,\, a_1 V_1 + a_2 V_2$. For any complex normed space $\mathcal{E}$ the norm of a map
out of $\ell^1(3)$ is the maximum of the norms of the images of the basis vectors,
since $\|\sum_i \alpha_i u_i\| \leq \bigl(\sum_i |\alpha_i|\bigr)\max_i \|u_i\|$ with
equality at the basis vectors. This gives the claimed equivalence, and in fact
$\|V\|_{X^* \to \mathcal{E}} = \max\{\|V_1\|, \|V_2\|, \|a_1V_1+a_2V_2\|\}$.
\end{proof}

\begin{lem}\label{lem:contraction}
Let $L_V:(\mathbb{C}^m, \|\cdot\|_A)^* \to (\mathbb{C}^{n\times n}, \|\cdot\|_{\ell_2^n \to \ell_2^n})$ be a linear operator of the form
\[
L_V(z_1, \ldots ,z_m) = z_1 V_1 + \cdots + z_m V_m,\quad V_1, \ldots , V_m \in \mathbb{C}^{n\times n}.
\]
The operator $L_V$ is contractive if and only if $\displaystyle\sup_{\|x\|_2 = 1=\|y\|_2} \|\mathbb{V}_{x,y} \|_{\ell_2^d \to \ell_2^d} \leq 1$, where
\[
\mathbb{V}_{x,y}:= \Big(\!\!\Big( \sum_{k=1}^m A_k(i,j) \langle V_k x , y\rangle \Big)\!\!\Big)_{i,j=1}^d.
\]
Equivalently, $L_V$ is contractive if and only if $(\langle V_1 x, y\rangle, \ldots , \langle V_m x, y\rangle) \in B_{X_A}$ for all euclidean unit vectors $x,y$.
\end{lem}
\begin{proof}
We compute the operator norm of $L_V$:
\begin{align*}
\|L_V\| &= \sup \big\{\|L_V(w) \|: w \in (\mathbb{C}^m, \|\cdot\|_A)^*_1\big\}\\
&=\sup_{\|C\|_{\operatorname{tr}} \leq 1} \Big\| \sum_{k=1}^m V_k \Big(\sum_{i,j=1}^d A_k(i,j) \overline{C(i,j)}\Big)\Big\|\\
&=\sup_{\|C\|_{\operatorname{tr}} \leq 1} \sup_{\|x\|_2=1=\|y\|_2} \sum_{i,j=1}^d \overline{C(i,j)} \Big(\sum_{k=1}^m A_k(i,j) \langle V_k x, y\rangle \Big)\\
&=\sup_{\|x\|_2=1=\|y\|_2} \sup_{\|C\|_{\operatorname{tr}} \leq 1} \operatorname{tr} \big(\mathbb{V}_{x,y}\, C^*\big)\\
&= \sup_{\|x\|_2=1=\|y\|_2} \|\mathbb{V}_{x,y}\|_{\ell_2^d \to \ell_2^d}.
\end{align*}
The equivalent statement follows directly: for every $(\lambda_1, \ldots , \lambda_m) \in (\mathbb{C}^m, \|\cdot\|_A)^*_1$, we have $\|\lambda_1 V_1 + \cdots + \lambda_m V_m\| \leq 1$. So, for any pair of euclidean unit vectors $x,y$, 
\[(\langle V_1 x, y\rangle , \ldots , \langle V_m x, y\rangle) \in B_{X_A},\] and conversely.
\end{proof}

\begin{cor}\label{cor:diagonal}
Let $L_V:(\mathbb{C}^m, \|\cdot\|_A)^* \to (\mathbb{C}^{n\times n}, \|\cdot\|_{\ell_2^n \to \ell_2^n})$ be as above, and assume $A_1, \ldots , A_m$ are $d\times d$ diagonal matrices. Then $L_V$ is contractive if and only if
\[
\max \Big\{\Big\|\sum_{k=1}^m A_k(1,1) V_k \Big\|, \ldots , \Big\| \sum_{k=1}^m A_k(d,d) V_k\Big\| \Big\} \leq 1.
\]
\end{cor}
\begin{rem}\label{rem:notequal}
Let $(P_1, \ldots, P_m)$ be an $m$-tuple of $d \times d$ matrices, and let
$\mathbb{P}_V := \sum_{k=1}^m P_k \otimes V_k$, acting on
$\mathbb{C}^d\otimes\mathbb{C}^n\cong \ell_2^{dn}$. Define
\[
(\mathbb{P}_V)_{x,y} := \Big(\!\!\Big( \sum_{k=1}^m P_k(i,j) \langle V_k x, y \rangle \Big)\!\!\Big)_{i,j=1}^d.
\]
Writing $R_x:\mathbb{C}^d\to\mathbb{C}^d\otimes\mathbb{C}^n$, $u\mapsto u\otimes x$, one checks that
$(\mathbb{P}_V)_{x,y}=R_y^*\,\mathbb{P}_V\,R_x$, and $R_x,R_y$ are isometries when
$\|x\|_2=\|y\|_2=1$. Hence
\[
\sup_{\|x\|_2=1=\|y\|_2} \|(\mathbb{P}_V)_{x,y}\|_{\ell_2^d \to \ell_2^d} \leq \|\mathbb{P}_V\|_{\ell_2^{dn} \to \ell_2^{dn}}.
\]
Consider the special case $P_k = A_k$ for $k = 1, \ldots, m$, and write
$\mathbb{A}_V = \sum_{k=1}^m A_k \otimes V_k$.
 
Suppose first that $A_1, \ldots, A_m$ are simultaneously unitarily diagonalizable.
A simultaneous unitary conjugation $A_k\mapsto U^*A_kU$ replaces $\mathbb{A}_V$ by
$(U^*\otimes I)\mathbb{A}_V(U\otimes I)$ and $(\mathbb{A}_V)_{x,y}$ by
$U^*(\mathbb{A}_V)_{x,y}U$, so it changes neither side; we may therefore assume the
$A_k$ are diagonal. Then $\mathbb{A}_V$ is block diagonal with blocks
$\sum_k A_k(i,i)V_k$, while $(\mathbb{A}_V)_{x,y}$ is diagonal with entries
$\big\langle \big(\sum_k A_k(i,i)V_k\big)x,y\big\rangle$. Hence both sides equal
$\max_{1\leq i\leq d}\big\|\sum_k A_k(i,i)V_k\big\|$ and
\begin{equation}\label{eq:diagequal}
\sup_{\|x\|_2=1=\|y\|_2} \|(\mathbb{A}_V)_{x,y}\|_{\ell_2^d \to \ell_2^d} = \|\mathbb{A}_V\|_{\ell_2^{dn} \to \ell_2^{dn}} .
\end{equation}
 
When moreover $m=d=2$, considerably more is true. Writing
$\phi_i(z)=z_1A_1(i,i)+z_2A_2(i,i)$, $i=1,2$, we have
$\|z\|_A=\max\{|\phi_1(z)|,|\phi_2(z)|\}$, and $\phi_1,\phi_2$ are linearly independent
precisely because $A_1,A_2$ are; so $X_A$ is isometric to $\ell_\infty^2$. And\^o's
theorem then gives $\alpha(X_A)=\alpha(X_A^*)=1$: every contractive $L_V$ is completely
contractive, so that $\|\sum_k P_k\otimes V_k\|\leq1$ for \emph{every} admissible test
tuple $P$, not merely for $P=A$. This is strictly stronger than \eqref{eq:diagequal},
and it is the reason one must look outside this case for a counterexample.
 
The tuple $A$ is in any event an admissible test tuple: at the $d$-th matrix level
\[
\Big\|\sum_{j=1}^m A_j\otimes \epsilon_j\Big\|_{\mathcal{M}_d(\operatorname{MIN} (X_A^*))}
=\sup_{z\in B_{X_A}}\Big\|\sum_{j=1}^m z_jA_j\Big\|=\sup_{\|z\|_A\leq 1}\|z\|_A=1,
\]
so $\|\mathbb{A}_V\|>1$ certifies that $L_V$ is not completely contractive.
Consequently, if some tuple $(V_1,\ldots,V_m)$ satisfies
\begin{equation} \label{notequal}
\sup_{\|x\|_2=1=\|y\|_2} \|(\mathbb{A}_V)_{x,y}\|_{\ell_2^d \to \ell_2^d} < \|\mathbb{A}_V\|_{\ell_2^{dn} \to \ell_2^{dn}},
\end{equation}
then, after dividing every $V_k$ by the (necessarily positive) left-hand side --- both
sides of \eqref{notequal} being homogeneous of degree one in $V$ --- we obtain a tuple
for which $L_V$ is contractive, by Lemma~\ref{lem:contraction}, while
$\|A_1 \otimes V_1 + \cdots + A_m \otimes V_m\| > 1$. This would provide an example of
a contractive linear map $L_V$ on $(\mathbb{C}^m, \|\cdot\|_A)^*$ that is not completely
contractive. The computations of \cite{MPV} show that strict inequality in
\eqref{notequal} does hold when $m = d = 2$ and $A_1, A_2$ are not simultaneously
diagonalizable.
 
It should be stressed, however, that \eqref{eq:diagequal} already excludes the tuples
of Subsection 1.2, which are diagonal: for those $A_1,A_2$ the inequality
\eqref{notequal} is never strict, whatever $V$ is chosen, even though
$\alpha(X_A)>1$ by Corollary~\ref{cor:main}. The route through the test tuple $P=A$
therefore cannot produce the counterexample of Section~\ref{sec:counterexample}, and
this is why Section~3 works with general test tuples $P\in\mathcal{M}_k$.
\end{rem}

\subsection{Definition of the subspaces of $(\mathbb{C}^3, \|\cdot\|_\infty$)}
Now let $(\mathbb{C}^2,\|\cdot\|_A)$ be the operator space induced by the map $\Phi_A: \mathbb{C}^2 \rightarrow \mathbb{C}^{m \times m}$ given by $\Phi(z_1, z_2)=z_1 A_1+z_2 A_2$, where $(A_1, A_2)$ are $m\times m$ diagonal matrices. Thus $\Phi$ is a complete isometry into a commutative $C^*$-algebra isometric to $\ell_\infty^m$ by construction.

In \cite{AFJS95}, the particular case of $m=3$ and
\[
(A_1, A_2) = \left(\begin{pmatrix} 1 & 0 & 0\\ 0&0&0\\ 0 & 0 & a_1 \end{pmatrix}, \begin{pmatrix} 0&0&0\\0&1&0\\0&0&a_2  \end{pmatrix}\right)
\]
was considered. With this choice, $\|(z_1,z_2)\|_A =\max\{|z_1|, |z_2|, |a_1 z_1 + a_2 z_2|\}$. It was noted in \cite{AFJS95} that if $a_1, a_2$ are chosen such that $|a_1|, |a_2| \leq 1$ while $|a_1| + |a_2| > 1$, then the norm $\| \cdot \|_A$ is not isometric to the $\ell^\infty$ norm, see the proof of Proposition 4.4 in \cite{AFJS95}. We state this in the form of a lemma below and provide an elementary proof. 
\begin{lem}
Let $\mathcal S\subseteq\ell_\infty^3$ be a $2$-dimensional complex subspace. Then $\mathcal S$ is not isometric to $\ell_\infty^2$ if and only if $\mathcal S$ is isometric to $(\mathbb{C}^2,\|\cdot\|_A)$ for some $a_1,a_2$ with $|a_1|\le 1$, $|a_2|\le 1$ and $|a_1|+|a_2|>1$.
\end{lem}
\begin{proof}
    Let $\psi_1,\psi_2,\psi_3\in\mathcal S^*$ be the restrictions of the three coordinate functionals, so that $\|x\|=\max_i|\psi_i(x)|$ and $\|\psi_i\|\le1$.

($\Leftarrow$) If $\mathcal S=(\mathbb{C}^2,\|\cdot\|_A)$ with $|a_1|+|a_2|>1$, pick unimodular $z_1,z_2$ aligning $a_1z_1$ and $a_2z_2$; then $\|(z_1,z_2)\|_A=|a_1|+|a_2|>1=\|(z_1,z_2)\|_\infty$, and one checks the unit ball of $\|\cdot\|_A$ has more than four extreme ``faces", so it is not isometric to $\ell_\infty^2$. 

($\Rightarrow$) Suppose $\mathcal S\not\cong\ell_\infty^2$.

\subsubsection*{Each $\|\psi_i\|=1$\textup{:}} If, say, $\|\psi_3\|<1$ then by compactness $|\psi_3(x)|<\|x\|$ for all $x\ne0$, so $\|x\|=\max(|\psi_1(x)|,|\psi_2(x)|)$ and $x\mapsto(\psi_1 x,\psi_2 x)$ is an isometry onto a $2$-dimensional subspace of $\ell_\infty^2$, i.e. onto $\ell_\infty^2$ — excluded.

\subsubsection*{Any two of the $\psi_i$ are independent\textup{:}} If $\psi_2=c\psi_1$ then $|c|=1$ and $\|x\|=\max(|\psi_1x|,|\psi_3x|)$, giving $\mathcal S\cong\ell_\infty^2$ again.

So fix the basis $\psi_1,\psi_2$ of $\mathcal S^*$, identify $x\leftrightarrow (z_1,z_2)=(\psi_1x,\psi_2x)$, and write $\psi_3=a_1\psi_1+a_2\psi_2$ with $a_1,a_2\ne0$. Then $\|(z_1,z_2)\|=\max\{|z_1|,|z_2|,|a_1z_1+a_2z_2|\}$, and $|a_1|+|a_2|>1$ (otherwise the third term is dominated by the first two and $\mathcal S\cong\ell_\infty^2$).

It remains to arrange $|a_1|,|a_2|\le1$, and for that we are free to re-choose which pair of the $\psi_i$ is the basis. Put $u=|a_1|,v=|a_2|>0$. Using the pair $(\psi_1,\psi_3)$ one has $\psi_2=-\tfrac{a_1}{a_2}\psi_1+\tfrac1{a_2}\psi_3$, with coefficient moduli $(u/v,\,1/v)$; using $(\psi_2,\psi_3)$ one gets $(1/u,\,v/u)$. Hence

the pair $(\psi_1,\psi_2)$ works iff $u\le1$ and $v\le1$;
the pair $(\psi_1,\psi_3)$ works iff $v\ge1$ and $v\ge u$;
the pair $(\psi_2,\psi_3)$ works iff $u\ge1$ and $u\ge v$.

If the first fails then $\max(u,v)>1$; if $u\ge v$ the third applies, and if $v\ge u$ the second does. So some pair always works. Finally write $a_j=r_je^{i\theta_j}$ and apply $w_j=e^{i\theta_j}z_j$ to reach $E_{r,s}$ with $0<r,s\le1$, $r+s>1$.
\end{proof}

\section{Extreme points}
Let $\mathcal{H}$ be an $n$-dimensional Hilbert space and $\mathcal{L}(\mathcal{H})$ the algebra of all linear transformations on $\mathcal{H}$. Let $B_n[0,1]$ be the closed unit ball in $\mathcal{L}(\mathcal{H})$ with the operator norm.

\begin{defn}[Admissible perturbations]\label{def:admissible}
An element $\Delta \in \mathbb{C}^{n\times n}$ is an \emph{admissible perturbation} of $T\in B_n[0,1]$ if there exists $\epsilon_0 > 0$ such that $\|T+ \epsilon \Delta\| \leq 1$ for all $|\epsilon| \leq \epsilon_0$. The \emph{tangent cone} $\mathcal{T}_T$ is the set of all admissible perturbations of $T$.
\end{defn}

Any $T\in B_n[0,1]$ admits a polar decomposition $T= U |T|$, where $U$ is unitary and $|T|:= (T^* T)^{1/2}$. Since $T$ is a contraction, the eigenvalues of $|T|$ lie in $[0,1]$. Let $E$ be the projection onto the eigenspace of $|T|$ corresponding to the eigenvalue $1$ and $F$ be the projection onto the eigenspace of $|T^*|$ corresponding to the eigenvalue $1$. Since the non-zero eigenvalues of $|T|$ and $|T^*|$ coincide, $\dim \operatorname{ran}(E)=\dim \operatorname{ran}(F)=k$ (say) for some $0\leq k \leq n$.

\begin{lem} \label{lem:3.5}
For any $T\in B_n[0,1]$, $\Delta\in \mathcal{T}_T$ if and only if $\Delta E = 0$ and $F \Delta=0$.
\end{lem}
\begin{proof}
\textit{Necessity.} Suppose $\Delta\in \mathcal{T}_T$ and $x\in \operatorname{ran}(E)$ with $\|x\|_2=1$. Then $|T|x = x$ and $\|Tx\| = \|U|T|x\| = \|Ux\|=1$. For all $|\epsilon| \leq \epsilon_0$,
\[
\|(T+ \epsilon \Delta)x\| = \|Tx + \epsilon \Delta x\| \leq 1.
\]
Since $\|Tx\|=1$ is an extreme value of the norm on the unit ball in $\ell_n^2$, we must have $\Delta x= 0$ for all $x\in \operatorname{ran}(E)$, hence $\Delta E = 0$. Second, $\|(T^* + \bar{\epsilon} \Delta^*)x\| \leq 1$ shows $\Delta^* \in \mathcal{T}_{T^*}$, and applying the same argument to $T^*$ gives $\Delta^* F = 0$, hence $F \Delta = 0$.

\textit{Sufficiency.} Suppose $\Delta E=0$ and $F\Delta=0$. The space $\mathcal{H}$ decomposes as
\begin{equation}\label{decompH}
\operatorname{ran}(E) \oplus \operatorname{ran}(E)^\perp = \mathcal{H} = \operatorname{ran}(F) \oplus \operatorname{ran}(F)^\perp.
\end{equation}
With respect to these decompositions (using $E$ in the domain and $F$ in the range), the operator $T$ has block form $T = \begin{pmatrix} T_{11} & 0 \\ 0 & T_{22} \end{pmatrix}$, where $T_{11}: \operatorname{ran}(E) \to \operatorname{ran}(F)$ is unitary and $T_{22}: \operatorname{ran}(E)^\perp \to \operatorname{ran}(F)^\perp$ is a strict contraction with $\|T_{22}\| = r < 1$.

Indeed, if $x\in \operatorname{ran}(E)$, then $TT^*(Tx) = T(T^*Tx) = Tx$, so $Tx \in \operatorname{ran}(F)$, giving $T_{12}=0$. For $y\in \operatorname{ran}(E)^\perp$, the map $U$ sends the $\lambda$-eigenspace of $|T|$ (for $\lambda<1$) into the $\lambda$-eigenspace of $|T^*|$, so $Ty\in \operatorname{ran}(F)^\perp$, giving $T_{21}=0$.

The conditions $\Delta E=0$ and $F\Delta=0$ together force $\Delta = \begin{pmatrix} 0 & 0 \\ 0 & \Delta_0 \end{pmatrix}$, where $\Delta_0: \operatorname{ran}(E)^\perp \to \operatorname{ran}(F)^\perp$. For $x = x_1 \oplus x_2$ with $x_1\in \operatorname{ran}(E)$ and $x_2\in \operatorname{ran}(E)^\perp$,
\[
\|(T+\epsilon \Delta) x\|^2 = \|Tx_1\|^2 + \|(T_{22}+\epsilon \Delta_0)x_2\|^2 = \|x_1\|^2 + \|(T_{22}+\epsilon \Delta_0)x_2\|^2.
\]
Since $\|T_{22}\|=r<1$, by continuity of the operator norm, there exists $\epsilon_0 > 0$ such that $\|T_{22}+\epsilon \Delta_0\| \leq 1$ for $|\epsilon| \leq \epsilon_0$. Therefore $\|(T+\epsilon \Delta) x\|^2 \leq \|x_1\|^2+\|x_2\|^2 = \|x\|^2$.
\end{proof}

\subsection{Extreme points of $\mathscr{V}$}\label{sec:extremeV}
Throughout this subsection $a_1,a_2$ are fixed complex numbers with $|a_1|\leq 1,$ $|a_2|\leq 1$
and $|a_1|+|a_2|>1$; note that these force $a_1\neq 0\neq a_2$.
Let $\mathscr{V}\subset \mathcal{M}_n \times \mathcal{M}_n$ be the set
\[
\mathscr{V} = \{(V_1, V_2) \mid \|V_1\| \leq 1,\, \|V_2\| \leq 1,\, \|V_3\| \leq 1 \},
\qquad V_3:=a_1 V_1+ a_2 V_2 .
\]
By Lemma \ref{lem:dual-ball-and-contraction}(ii), applied with $\mathcal{E}=\mathcal{M}_n$, $\mathscr{V}$
is precisely the set of pairs of $n\times n$ matrices that induce a $\mathcal M_n$-valued  contractive linear map on
the dual of the normed linear space $(\mathbb{C}^2, \|\cdot\|_A)$.
Clearly, $\mathscr{V}$ is non-empty, convex, compact and balanced.
When $a_1=r$ and $a_2=s$ are positive reals we write
\begin{equation}\label{eq:Vrs}
\mathscr{V}_{r,s}(n) := \{(V_1,V_2)\in\mathcal{M}_n\times\mathcal{M}_n: \|V_1\|\leq 1,\ \|V_2\|\leq 1,\ \|rV_1+sV_2\|\leq 1\}.
\end{equation}
 
We have $\mathscr{V} = \mathcal{B}_1 \cap \mathcal{B}_2 \cap \mathcal{B}_3$, where
\begin{enumerate}
\item $\mathcal{B}_1= \{(V_1,V_2): \|V_1\| \leq 1\}=B_n[0,1]\times \mathcal{M}_n$,
\item $\mathcal{B}_2= \{(V_1,V_2): \|V_2\| \leq 1\} = \mathcal{M}_n \times B_n[0,1]$,
\item $\mathcal{B}_3= \{(V_1,V_2): \|a_1V_1+a_2 V_2\| \leq 1\}$.
\end{enumerate}
 
Fix $(V_1,V_2)\in\mathscr{V}$. For $i=1,2,3$, let $E_i$ be the spectral projection onto
the eigenspace of $|V_i|$ with eigenvalue $1$, and $F_i$ the spectral projection onto
the eigenspace of $|V_i^*|$ with eigenvalue $1$.

\begin{prop}\label{prop:tangent}
The tangent cone of a pair $(V_1,V_2)$ in $\mathscr{V}$ is
\[
\mathcal{T}_{(V_1,V_2)} =\big\{ (\Delta_1,\Delta_2)\mid F_1\Delta_1 = 0 = \Delta_1 E_1,\; F_2\Delta_2=0=\Delta_2 E_2,\; F_3 \Delta_3 =0= \Delta_3 E_3\big\}.
\]
\end{prop}

\begin{proof}
Directly from Definition \ref{def:admissible}, a pair $(\Delta_1,\Delta_2)$ is an
admissible perturbation of $(V_1,V_2)$ in $\mathcal{B}_1\cap\mathcal{B}_2\cap\mathcal{B}_3$
if and only if it is one in each $\mathcal{B}_i$ separately: one direction is trivial,
and for the other it suffices to take the smallest of the three constants $\epsilon_0$.
Likewise $(\Delta_1,\Delta_2)$ is admissible in $\mathcal{B}_1=B_n[0,1]\times\mathcal{M}_n$
if and only if $\Delta_1$ is admissible for $V_1$ in $B_n[0,1]$, and similarly for
$\mathcal{B}_2$; while it is admissible in $\mathcal{B}_3$ if and only if there is
$\epsilon_0>0$ with $\|V_3+\epsilon\Delta_3\|\leq 1$ for all $|\epsilon|\leq\epsilon_0$,
that is, if and only if $\Delta_3$ is admissible for $V_3$. Now apply
Lemma~\ref{lem:3.5} three times.
\end{proof}

\begin{thm} \label{thm:3.8}
A pair $(V_1,V_2)\in \mathscr{V}$ is an extreme point if and only if the tangent cone $\mathcal{T}_{(V_1,V_2)}$ contains only the pair $(0,0)$.
\end{thm}

\begin{cor}
It follows from Theorem~\ref{thm:3.8} that if any two of the three contractions $V_1,V_2,V_3$ are unitary, then $(V_1,V_2)$ must be an extreme point of $\mathscr{V}$. 
\end{cor}
\begin{proof}
If $V_3$ and one of $V_1, V_2$ are unitary, then $\Delta_3=0$ (since $V_3$ is an extreme point of the usual unit ball), so $a_1 \Delta_1 + a_2\Delta_2 =0$. But if one of $V_1$ or $V_2$ is unitary, the corresponding $\Delta$ vanishes, forcing both to vanish. If both $V_1$ and $V_2$ are unitary, both $\Delta_1$ and $\Delta_2$ must vanish immediately.
\end{proof}

\begin{lem}\label{existnonnormalext} For every $0<r,s\leq 1$, $r+s > 1$,
there exists a nonnormal matrix $Y_t\in \mathcal{M}_2$ such that $(I,Y_t)$ is an extreme point of $\mathscr V_{r,s}(2).$
\end{lem}
\begin{proof} Recall that $\Big\|\begin{pmatrix} a & b\\ 0 & a \end{pmatrix}\Big\| \leq 1$ if and only if $|b| \leq 1-|a|^2$. 

    Let $Y_t = \begin{pmatrix} t&1-t^2 \\ 0 & t \end{pmatrix}$. Choosing $|t| < 1$, the norm $\|Y_t\| = 1$. In this case, $Y_t$ is not normal. 
    So, we must have $-1<t<\frac{1-r}{s}<1$. Also, choosing 
     $t$ such that $1-(r+st)^2 = s(1-t^2)$ we have $\|r I + sY_t\| = 1$. Thus, the pair $(I,Y_t)$ is in $\mathscr{V}_{r,s}$ for $t$ such that 
    \begin{equation}\label{choice of t} -1<t<\frac{1-r}{s}<1, \,\, \text{and } 1-(r+st)^2 = s(1-t^2).\end{equation}

We claim the function $g(t):=1-(r+st)^2-s(1-t^2)$ has a root in $\bigl(-1,\tfrac{1-r}{s}\bigr)$. Indeed
\[
g(-1)=1-(r-s)^2>0 ,\qquad
g\!\left(\tfrac{1-r}{s}\right)=-s\Bigl(1-\tfrac{(1-r)^2}{s^2}\Bigr)<0 ,
\]
the first because $|r-s|<1$ (either $r-s=1$ or $s-r=1$ would force one of $r,s$ to exceed $1$), the second because $\bigl|\tfrac{1-r}{s}\bigr|<1$. 

If $s=1,$ then the unique root $t=-r/2$ of 
$g(t)=-2rt-r^2$ lies in $(-1,1-r),$ since $0<r\leq 1$.
If $0<s<1$ the two roots of $g$ are
\[
t_\pm=\frac{r\pm\sqrt{\dfrac{r^2-(1-s)^2}{s}}}{1-s},
\]
the discriminant being positive because
$r^2s^2-s(1-s)(1-s-r^2)=s\bigl[r^2-(1-s)^2\bigr]=s(r+s-1)(r+1-s)>0$.
Since $t_+>\frac{r}{1-s}>1$, the root in $\bigl(-1,\tfrac{1-r}{s}\bigr)$ must be $t_-$.
Fix such a $t$ for the rest of the proof.

Put $c= r+st$. Because $g(t)=0$ we have $s(1-t^2)=1-c^2$, and therefore
\[
rI + s Y_t=\begin{pmatrix} c& 1-c^2 \\ 0 & c\end{pmatrix}=Y_c ,
\]
with $|c|<1$. The vectors $e_t=\binom{t}{1}$ and $e_c=\binom{c}{1}$ are linearly
independent, since
\[c-t=r-(1-s) t\geq r-(1-s)=r+s-1>0,\]
using $t<1$ and $1-s\geq 0$. Moreover $Y_te_t=\binom{1}{t}$ and $Y_ce_c=\binom{1}{c}$,
so that
\[\|Y_t e_t\|= \|e_t\| \quad\text{and}\quad \|Y_c e_c\|=\|e_c\|,\]
that is, $e_t$ and $e_c$ are norm-attaining for the contractions $Y_t$ and $Y_c$, and
$\mathbb{C} e_t=\ker(I-Y_t^*Y_t)$, $\mathbb{C} e_c=\ker(I-Y_c^*Y_c)$ (each kernel is one
dimensional because $\det Y_t=t^2$ and $\det Y_c=c^2$ have modulus less than one).

Let $(\Delta_1,\Delta_2)$ be an admissible perturbation of $(I,Y_t)$. Since $I$ is
unitary, $\Delta_1=0$, and $\Delta_3=s\Delta_2$. Writing $\Delta:=\Delta_2$,
Proposition~\ref{prop:tangent} gives $\Delta e_t=0=s\Delta e_c$. As $e_t, e_c$ span
$\mathbb{C}^2$ and $s>0$, we get $\Delta = 0$. By Theorem~\ref{thm:3.8}, $(I,Y_t)$ is an
extreme point of $\mathscr{V}_{r,s}(2)$.
\end{proof}
\section{Contractive linear maps and homomorphisms} 
Let $\mathbb{B}$ be a ball in $\mathbb{C}^m$ with respect to some norm and let $\operatorname{H}^\infty(\mathbb{B})$ be the algebra of bounded holomorphic functions.  Note that an element $P$ of $\mathcal{M}_k \otimes \operatorname{H}^\infty(\mathbb{B})$ is a holomorphic function 
from $\mathbb{B}$ to $\mathcal{M}_k$ and therefore has a power series expansion of the form 
\[
P(z) = P(0) + z_1 P_1 + \cdots + z_m P_m + \text{ higher order terms}.
\]
Thus the derivative $DP(0)$ is a linear transformation from $\mathbb{C}^m$ to $\mathcal{M}_k$ and is of the form $(P_1, \dots , P_m)$, that is, $P_i = (\partial_i P)(0)$, $i=1, \ldots ,m$. For any fixed but arbitrary $w\in \mathbb{B}$ and a $m$-tuple $V_1, \ldots, V_m$ in $End(\mathbb{C}^n)$,
define the homomorphism $\varrho_V^{(k)}:\mathcal{M}_k \otimes \operatorname{H}^\infty(\mathbb{B}) \to \mathcal{M}_{2kn}$ by setting   
\[\varrho^{(k)}_V(P) = \begin{pmatrix}
P(0)\otimes I_n & \sum_{j=1}^m P_j\otimes V_j\\
0 & P(0)\otimes I_n
\end{pmatrix}, \quad P \in \mathcal{M}_k \otimes \operatorname{H}^\infty(\mathbb{B}).\]
There is a unique $C^*$-norm on the algebra $\mathcal{M}_k \otimes \operatorname{H}^\infty(\mathbb B)$, namely, 
\[\|P\|_\infty:= \sup\{\|P(z)\|_{\ell_2^k \to \ell_2^k}: z\in \mathbb{B}\}.\]

There is also a linear map $L^{(k)}_V: \mathcal{M}_k \otimes \operatorname{H}^\infty(\mathbb{B}) \to \mathcal{M}_{kn}$
defined by setting 
\[L_V^{(k)}(P) = \sum_{j=1}^m P_j\otimes V_j,\,\, P\in \mathcal{M}_k \otimes \operatorname{H}^\infty(\mathbb{B}).\]

For any $k\in \mathbb{N}$, the contractivity of the homomorphism $\varrho^{(k)}$ is equivalent to the contractivity of $L_V^{(k)}$. One direction of this equivalence is evident, the other direction is the zero lemma stated below, {{
see \cite[Lemma 3.3]{GM} and  \cite[Lemma 1.6]{GMNS}.}}

\begin{lem} \label{lem:rho-vs-L}
A contractive homomorphism $\varrho$, $\varrho := \varrho^{(1)}$, is completely contractive, that is, all the homomorphisms $\varrho^{(k)}$ are contractive for all $k > 1$ if and only if the same is true of the linear map $L_V$, $L_V:=L_V^{(1)}$. 
\end{lem}
A version of generalized Schwarz lemma (see, \cite[Proposition 14.5]{Vern}) shows that $\mathcal{M}_k \otimes \mathbb{C}^m$ equipped with these norms is exactly the MIN norm on $(\mathbb{C}^m, \|\cdot\|)^*$.  
To study which contractive linear maps are completely contractive, Paulsen introduced
in \cite{vip} the quantity
\begin{equation}\label{eq:alpha}
\alpha(X)=\|\mathrm{id}:\mathrm{MIN}(X)\to\mathrm{MAX}(X)\|_{cb}
\end{equation}
and proved the following.
\begin{prop}[Proposition 2.3, \cite{vip}]Let $\mathbb{B}$ be any unit ball in $\mathbb{C}^m$ with respect to some norm $\|\cdot\|$, $\mathbb{B}^*$ be the unit ball of $(\mathbb{C}^m,\|\cdot\|)^*$. Then 
\[\alpha(X)=\sup \bigg \{\Big \|\sum_{i=1}^m P_i \otimes V_i\Big \|:\Big \|\sum_{i=1}^m \lambda_i V_i\Big \| \leq 1,\Big\|\sum_{i=1}^m \mu_i P_i\Big\| \leq 1, \lambda \in \mathbb{B},\mu \in \mathbb{B}^*\bigg \},\]
where $\lambda = (\lambda_1, \ldots, \lambda_m)$ and $\mu =(\mu_1, \ldots, \mu_m)$
\end{prop}
Fix a pair $a_1, a_2$ of complex numbers such that $|a_1|\leq 1, |a_2| \leq  1$ and $|a_1|+|a_2| > 1$. Let  
\begin{equation} \label{eq:Defn-norm}
\|(z_1,z_2)\|_A = \max\{|z_1|, |z_2|, a_1z_!+a_2z_2|\}.
\end{equation}
Write $a_j=r_j e^{i\theta_j}$, and put $r=r_1$, $s=r_2$. The coordinate change $w_j=e^{i\theta_j}z_j$ reduces the norm $\|\cdot\|_A$ to
\[
\|(z_1,z_2)\|_{r,s}:
=\max\{|z_1|,\ |z_2|,\ |rz_1+sz_2|\}.
\]
By hypothesis, $0<r,s\leq 1$, $r+s>1$. In what follows, we let $E_{r,s}$ denote the normed linear space $(\mathbb{C}^2, \|\cdot\|_{r,s})$. 
\begin{thm} \label{mainthm}
Let $0<r,s\leq1$ with $r+s>1$, and let $Y\in \mathcal{M}_n$ satisfy
$\|Y\|\leq1$ and $\|rI+sY\|\leq 1$. Suppose that
\begin{equation}\label{eq:spanning}
\ker(I-Y^*Y)+\ker\bigl(I-(rI+sY)^*(rI+sY)\bigr)=\mathbb{C}^n,
\end{equation}
equivalently, that $\mathcal{H}=\mathbb{C}^n$ has a basis $h_1,\ldots,h_n$ such that for each
$j$ at least one of
$\|Yh_j\|=\|h_j\|$, $\|(rI+sY)h_j\|=\|h_j\|$ holds. Then
\[
L_{(I,Y)}:\operatorname{MIN}(E_{r,s}^*)\longrightarrow \mathcal{M}_n
\]
is completely contractive if and only if $Y$ is normal.
\end{thm}
\begin{proof}
    Put
\[
T_1=\begin{pmatrix}0&I_{\mathcal{H}}\\0&0\end{pmatrix},
\qquad
T_2=\begin{pmatrix}0&Y\\0&0\end{pmatrix}
\]
on $\mathcal{H}\oplus \mathcal{H}$. Let $K$ be the closed unit ball of $E_{r,s}$ and
$P(K)$ the uniform closure of the polynomials on $K$. Because $T_iT_j=0$ for
$i,j=1,2$, the induced homomorphism is
\[
\varrho_{(I,Y)}(p)=
\begin{pmatrix}
p(0)I&
\partial_1p(0)I+\partial_2p(0)Y\\
0&p(0)I
\end{pmatrix},\quad p\in P(K).
\]

It is contractive if and only if the corresponding linear map is contractive, which by
Corollary~\ref{cor:diagonal} is precisely the condition
$\max\{\|I\|,\|Y\|,\|rI+sY\|\}\leq 1$.
 
Suppose that $\varrho_{(I,Y)}$ is completely contractive.
Then Arveson's boundary dilation theorem provides a commuting pair of normal operators
$N_1,N_2$ on a Hilbert space $\mathcal{K} \supseteq \mathcal{H}\oplus \mathcal{H}$ and
an isometry $W:\mathcal{H}\oplus \mathcal{H}\to\mathcal{K}$ with
\[
\varrho_{(I,Y)}(p)=W^*p(N_1,N_2)W, \qquad
\sigma(N_1,N_2)\subseteq\partial_S P(K)\subseteq K,
\]
where $\partial_S P(K)$ denotes the Shilov boundary \cite[Corollary 7.7]{Vern}.
Define $i_1, i_2:\mathcal{H} \to \mathcal{H} \oplus \mathcal{H}$ by
$i_1(h) = (h,0)$ and $i_2(h)=(0,h)$, and put
\[
J_1h=W(h,0),\qquad J_2h=W(0,h), \quad h\in \mathcal{H},
\]
so that $J_1= W\circ i_1$ and $J_2=W\circ i_2$. Consequently
$J_1^* J_1 = i_1^* W^* W i_1 = i_1^*i_1 = I_{\mathcal{H}}$, likewise
$J_2^*J_2=I_{\mathcal{H}}$, and $J_1^*J_2 = i_1^* i_2=0$.
Note $W(h,k)=J_1h+J_2k$, whence
\[W^*N_jW= \begin{pmatrix}
J_1^*N_jJ_1&J_1^*N_jJ_2\\
J_2^*N_jJ_1&J_2^*N_jJ_2
\end{pmatrix}, \quad j=1,2,
\]
and $W^*N_jW=T_j$, $j=1,2$, gives
\[
J_1^*N_1J_2=I,\qquad J_1^*N_2J_2=Y.
\]
For any $p\in P(K),$ we have, by the spectral mapping theorem for commuting normals,
\[
    \|p(N_1,N_2)\| = \sup\{|p(z)|: z\in \sigma(N_1, N_2)\}
    \leq \sup\{|p(z)|: z\in K\}.
\]
Taking $p=z_1$, $p=z_2$ and $p=rz_1+sz_2$ shows that $N_1$, $N_2$ and $r N_1 + s N_2$
are contractions.
 
From $J_1^*N_1J_2=I,$ we obtain
\begin{equation} \label{NandJ}
N_1J_2=J_1 \qquad\text{and}\qquad N_1^*N_1J_2=J_2.
\end{equation}
For the first, using $\|N_1J_2h\|\leq\|J_2h\|=\|h\|$, $\|J_1h\|=\|h\|$ and
$\langle J_1^*N_1J_2h,h\rangle=\|h\|^2$,
\begin{align*}
    \|N_1J_2h-J_1h\|^2 &= \|N_1J_2h\|^2 + \|J_1h\|^2 - 2 \operatorname{Re} \langle J_1^*N_1J_2 h , h \rangle\\
    &\leq \|h\|^2 + \|h\|^2 - 2\|h\|^2 = 0.
\end{align*}
Hence $\|N_1J_2h\|=\|J_1h\|=\|h\|=\|J_2h\|$ for every $h$, and since $I-N_1^*N_1\geq0$,
\[
\bigl\|(I-N_1^*N_1)^{1/2}J_2h\bigr\|^2 = \|J_2h\|^2-\|N_1J_2h\|^2 =0 ,
\]
so $(I-N_1^*N_1)J_2h=0$ for every $h$, which is the second identity in \eqref{NandJ}.
Put $D=I-N_1^* N_1$ and $\mathcal K_0=\ker D$.
 
\begin{theorem*}\label{claim:K0} \begin{itemize}[leftmargin=\dimexpr 24pt]
\item[\rm(i)] $\mathcal K_0$ is a common reducing subspace of $N_1, N_2$;
\item[\rm(ii)] $U:=N_1|_{\mathcal K_0}$ is unitary on $\mathcal K_0$;
\item[\rm(iii)] $N:=(N_1^*N_2)|_{\mathcal K_0}$ is normal, with $\|N\|\leq 1$ and
$\|rI+sN\|\leq 1$.
\end{itemize}
\end{theorem*}
 
\noindent\textit{Proof of \textup{(i)}.} Since $N_j N_1=N_1 N_j$ and $N_1$ is normal,
Fuglede's theorem gives $N_j N_1^* = N_1^* N_j$, $j=1,2$. Thus $N_j$ commutes with both
factors of $N_1^* N_1$, so
\[
D N_j =\left(I-N_1^* N_1\right) N_j = N_j-N_j N_1^* N_1 = N_j D .
\]
Because $D=D^*$, taking adjoint shows that $D$ commutes with $N_j^*$ as well.
Consequently, if $D h=0$ then $D N_j h=N_j D h=0$ and $D N_j^* h=N_j^* D h=0$, so
$\mathcal{K}_0$ is invariant under $N_j$ and $N_j^*$ for $j=1,2$, that is, it reduces both.
 
\noindent\textit{Proof of \textup{(ii)}.} By the definition of $\mathcal{K}_0,$ we have
$U^* U=I_{\mathcal{K}_0}$, and normality of $N_1$ gives
$U U^*=\left(N_1 N_1^*\right)|_{\mathcal{K}_0}=\left(N_1^* N_1\right)|_{\mathcal{K}_0}=I_{\mathcal{K}_0}$.
Hence $U$ is unitary.
 
\noindent\textit{Proof of \textup{(iii)}.} Set $V=N_2|_{\mathcal{K}_0}$, a normal
contraction commuting with $U$ and $U^*$. Then $N=U^* V$ satisfies
$N^* N=V^* U U^* V=V^* V$ and $N N^*=U^* V V^* U=V V^*=V^* V$, so $N$ is normal.
Moreover $\|N\|\leq\|V\|\leq 1$, and
$rI+sN=U^*(rU+sV)=U^*\bigl((rN_1+sN_2)|_{\mathcal{K}_0}\bigr)$ is a contraction because
$rN_1+sN_2$ is one and $U^*$ is unitary. \hfill$\square$
 
The identity $N_1^* N_1 J_2=J_2$ gives $J_2 \mathcal{H} \subseteq \mathcal{K}_0$. Since
$\mathcal K_0$ reduces $N_1$ and $N_1J_2=J_1$, also
$J_1\mathcal{H} = N_1J_2\mathcal{H} \subseteq\mathcal K_0$. The ranges of $J_1$ and
$J_2$ are orthogonal, so
$W(\mathcal{H}\oplus \mathcal{H})=J_1\mathcal{H}\oplus J_2\mathcal{H}\subseteq\mathcal K_0$;
hence we may replace $\mathcal{K}$ by the reducing subspace $\mathcal{K}_0$ without
disturbing the dilation. Using $N_1J_2=J_1$ once more,
\[
J_2^* N J_2 =J_2^* N_1^* N_2 J_2 =\left(N_1 J_2\right)^* N_2 J_2 =J_1^* N_2 J_2 = Y .
\]
 
Write $J=J_2$ and $P=J J^*$, the orthogonal projection onto $J \mathcal{H}$. Let $h$ be
one of the basis vectors $h_j$ furnished by the hypothesis.
 
If $\|Yh\|=\|h\|$, then, since $PNJh=JJ^*NJh=J(Yh)$ and $\|N\|\leq1$,
\[\|h\|=\|Y h\|=\|P N J h\| \leq\|N J h\| \leq\|h\| .\]
Equality throughout gives $\|PNJh\| = \|NJh\|$, so $NJh \in J \mathcal{H}$ and
$N J h=J Y h$.
 
If instead $\|(rI+sY)h\|=\|h\|$, apply the same argument to the contraction $rI+sN$,
whose compression is $J^*(rI+sN)J=rI+sY$:
\[
 \|h\|=\|P(r I+s N) J h\| \leq\|(r I+s N) J h\| \leq\|h\| .
\]
Hence $(r I+s N) J h \in J \mathcal{H}$ and
\[
(rI+sN)Jh= P(rI+sN)Jh = J(rI+sY)h .
\]
Expanding, $rJh+sNJh=rJh+sJYh$; since $s>0$, again $NJh=JYh$.
 
Thus $NJh_j=JYh_j$ for every $j$, and by linearity $NJh=JYh$ for every
$h\in \mathcal{H}$. Therefore $N(J\mathcal{H})\subseteq J\mathcal{H}$, so $J\mathcal{H}$
is invariant under $N$.
 
A finite-dimensional invariant subspace of a normal operator is reducing. Indeed,
represent a normal operator $R$ as $\begin{pmatrix} A & B \\0 & C \end{pmatrix}$
relative to an orthogonal decomposition $F\oplus G$ with $F$ finite dimensional and
invariant. Comparing the $(1,1)$ entries of $R^*R$ and $RR^*$ gives
$A^*A=AA^* + BB^*$ on $F$, and taking traces,
\[
\operatorname{tr}\left(B B^*\right)=\operatorname{tr}\left(A^* A\right)-\operatorname{tr}\left(A A^*\right)=0 .
\]
Since $B B^* \geq 0$, this forces $B=0$, so $F$ reduces $R$. Applying this with $R=N$
and $F=J\mathcal{H}$, the subspace $J\mathcal{H}$ reduces $N$.
 
Since $J:\mathcal H\to J\mathcal{H}$ is unitary, $Y=J^*(N|_{J\mathcal{H}})J$ is
unitarily equivalent to $N|_{J\mathcal{H}}$, which is normal. Hence $Y$ is normal.
 
For the converse, suppose $Y$ is normal, with eigenvalues $\lambda_1,\ldots,\lambda_n$.
Since $\|Y\| \leq 1$ and $\|rI + sY\| \leq 1$,
\[
|\lambda_j|\leq1,\qquad |r+s\lambda_j|\leq1, \qquad 1\leq j \leq n,
\]
so $(1,\lambda_j)\in B_{E_{r,s}}$ for every $j$. For arbitrary $S,T\in \mathcal{M}_k$,
unitary diagonalisation of $Y$ gives
\begin{align*}
\|S\otimes I+T\otimes Y\|
&=\max_j\|S+\lambda_jT\|
\leq
\sup_{(z_1,z_2)\in B_{E_{r,s}}}
\|z_1S+z_2T\| .
\end{align*}
The right-hand side is precisely the norm of the corresponding element of
$\mathcal{M}_k(\operatorname{MIN}(E_{r,s}^*))$, since $(E_{r,s}^*)^*=E_{r,s}$. As this holds for every $k$,
$L_{(I,Y)}$ is completely contractive, equivalently (Lemma~\ref{lem:rho-vs-L})
$\varrho_{(I,Y)}$ is completely contractive.
\end{proof}
We state below Theorem \ref{mainthm} in an equivalent form using the extreme points of the set $\mathscr{V}_{r,s}$.
\begin{thm} \label{mainthmext}
Let $0<r,s\leq1$ with $r+s>1$, and let $Y\in \mathcal{M}_n$ satisfy $\|Y\|\leq1$ and
$\|rI+sY\|\leq 1$. Suppose that $(I, Y)$ is an extreme point of
$\mathscr{V}_{r, s}(n)$. Then
\[
L_{(I, Y)}: \operatorname{MIN}\left(E_{r, s}^*\right) \longrightarrow \mathcal{M}_n
\]
is completely contractive if and only if $Y$ is normal.
\end{thm}
\begin{proof}
Write $V_3=rI+sY$ and
\[
M:=\ker(I-Y^*Y)+\ker(I-V_3^*V_3),\qquad
M':=\ker(I-YY^*)+\ker(I-V_3V_3^*),
\]
so that $M=\operatorname{ran} E_2+\operatorname{ran} E_3$ and $M'=\operatorname{ran} F_2+\operatorname{ran} F_3$ in the notation of
Subsection~\ref{sec:extremeV}. Since $V_1=I$ is unitary we have $E_1=F_1=I$, hence
$\Delta_1=0$ for every admissible perturbation, and $\Delta_3=s\Delta_2$ with $s>0$.
Therefore, by Proposition~\ref{prop:tangent}, writing $\Delta=\Delta_2$,
\[
\mathcal{T}_{(I,Y)}\cong\bigl\{\Delta\in\mathcal{M}_n:\ \Delta|_{M}=0,\ \operatorname{ran}\Delta\subseteq (M')^{\perp}\bigr\},
\]
a space naturally isomorphic to the space of linear maps $\mathbb{C}^n/M\to (M')^{\perp}$.
Hence
\[
\dim_{\mathbb{C}} \mathcal{T}_{(I,Y)}=(n-\dim M)(n-\dim M'),
\]
and by Theorem~\ref{thm:3.8} the pair $(I,Y)$ is extreme in $\mathscr{V}_{r,s}(n)$ if
and only if $M=\mathbb{C}^n$ \emph{or} $M'=\mathbb{C}^n$.
 
If $M=\mathbb{C}^n$, then hypothesis \eqref{eq:spanning} of Theorem~\ref{mainthm} holds and
the conclusion follows at once.
 
If instead $M'=\mathbb{C}^n$, we pass to adjoints. Since $r,s$ are real,
$(rI+sY)^*=rI+sY^*$, so $\|Y^*\|\leq 1$ and $\|rI+sY^*\|\leq1$, and
$M'$ is exactly the subspace $M$ formed from $Y^*$; thus $Y^*$ satisfies
\eqref{eq:spanning} and Theorem~\ref{mainthm} applies to $Y^*$. It remains to observe
two symmetries. First, $Y$ is normal if and only if $Y^*$ is. Second, $L_{(I,Y)}$ is
completely contractive if and only if $L_{(I,Y^*)}$ is: for $S,T\in\mathcal{M}_k$,
\[
\|S\otimes I+T\otimes Y\| = \|S^*\otimes I+T^*\otimes Y^*\| ,
\]
while $B_{E_{r,s}}$ is invariant under $z\mapsto\bar z$ (again because $r,s$ are
real), so that
\[
\sup_{z\in B_{E_{r,s}}}\|z_1S^*+z_2T^*\|=\sup_{z\in B_{E_{r,s}}}\|z_1S+z_2T\| ,
\]
that is, 
$\bigl\|L_{(I,Y)}^{(k)}\bigr\| = \bigl\|L_{(I,Y^*)}^{(k)}\bigr\|$, $k\geq 1$ and hence the two maps have the same matricial norms at every level. Applying
Theorem~\ref{mainthm} to $Y^*$ and transporting the conclusion back gives the result.
\end{proof}

\section{A counterexample} \label{sec:counterexample}
We can now record the counterexample. Recall the standing parameter assumptions
\[
\left|a_1\right| \leq 1, \quad\left|a_2\right| \leq 1, \quad\left|a_1\right|+\left|a_2\right|>1 .
\]
\begin{thm}\label{thm:counterexample}
For every pair of reals $0<r,s\leq 1$ with $r+s>1$ there is a contractive linear map
\[
L_{(I,Y_t)}:\operatorname{MIN}(E_{r,s}^*)\longrightarrow \mathcal{M}_2
\]
that is not completely contractive. Consequently $\alpha(E_{r,s}^*)>1$.
\end{thm}
\begin{proof}
Lemma~\ref{existnonnormalext} produces a non-normal $Y_t\in \mathcal{M}_2$ with
$(I,Y_t)\in\mathscr{V}_{r,s}(2)$, so $L_{(I,Y_t)}$ is contractive by
Lemma~\ref{lem:dual-ball-and-contraction}(ii); and the norm-attaining vectors
$e_t$ of $Y_t$ and $e_c$ of $rI+sY_t$ are linearly independent, so that
$\{e_t,e_c\}$ is a basis of $\mathbb{C}^2$ of the kind required in
Theorem~\ref{mainthm}. Since $Y_t$ is not normal, Theorem~\ref{mainthm} shows that
$L_{(I,Y_t)}$ is not completely contractive. The last assertion is immediate from
\eqref{eq:alpha}.
\end{proof}
 
For complex coefficients $a_j=\left|a_j\right| e^{i \theta_j}$, let $Y_\star$ be the
matrix $Y_t$ of Lemma~\ref{existnonnormalext} for the parameters $r=|a_1|$, $s=|a_2|$,
and take $Y=e^{i\left(\theta_2-\theta_1\right)} Y_\star$.
Then
\[
a_1 I+a_2 Y=e^{i \theta_1}\left(r I+s Y_\star\right),
\]
and since multiplication by a unimodular scalar changes neither $|Y|$, nor $|Y^*|$, nor
normality, nor any of the three norms involved, contractivity, extremality and
non-normality are all preserved.
\begin{cor}\label{cor:main}
   Let $\mathcal S$ be a two-dimensional complex subspace of $\ell_\infty^3$. If $\mathcal S$ is not complex-linearly isometric to $\ell_\infty^2$, then $\alpha(\mathcal S)>1$.
\end{cor}
\begin{proof} Assume without loss of generality that $\mathcal S$ is $E_{r,s}$ with
$0<r,s\leq1$, $r+s>1$.
Our discussion above gives a contractive but not completely contractive map on $\operatorname{MIN}(E_{r,s}^*)$, so $\alpha(E_{r,s}^*)>1$.
Since $\alpha(E_{r,s}) = \alpha(E_{r,s}^*)$, see \cite[Proposition 2.4]{vip}, it follows that 
\[
\alpha(\mathcal S)=\alpha(E_{r,s})
=\alpha(E_{r,s}^*)>1.
\qedhere\]
\end{proof}
\begin{rem}\label{rem:oss} Every Operator space structure of a normed linear space $X$ lies between MIN and MAX, and $\alpha(X)=\|\mathrm{id}:\mathrm{MIN}(X)\to\mathrm{MAX}(X)\|_{cb}$. 
Paulsen observes that $\alpha(X)=1$ precisely when the identity is a complete isometry from $\mathrm{MIN}(X)$ to $\mathrm{MAX}(X)$, equivalently when $X$ carries a unique operator space structure \cite[p.~190]{Vern}.
\end{rem} 
We may now restate Corollary \ref{cor:main} as follows.
\begin{thm} \label{cor:oss}
    A two-dimensional complex subspace of $\ell_\infty^3$ carries a unique operator space structure if and only if it is isometric to $\ell_\infty^2$.
\end{thm}
\begin{rem} Let $\mathcal{Q}$ be a two-dimensional complex quotient of $\ell_\infty^3$. If $\mathcal{Q}$ is not isometric to $\ell_\infty^2$, then $\alpha(\mathcal Q)>1$.

Since $\alpha\ge 1$ always, it suffices to prove the contrapositive: if
$\alpha(\mathcal{Q})=1$, then $\mathcal{Q}$ is isometric to $\ell_\infty^2$.
Suppose that $\alpha(\mathcal{Q})=1$. Then, in particular, $\mathcal{Q}$ has Property P. Therefore, as noted in \cite[Fact 1]{BM},  $\mathcal{Q}^*$ also has Property P, or equivalently, the $2$- summing property by \cite[Fact 2]{BM}. On the other hand, $\mathcal{Q}^*$ embeds isometrically into $\ell_1^3$, being the dual of a quotient of $\ell_\infty^3$. A two-dimensional subspace of $\ell_1^3$ with the $2$-summing property is isometric to $\ell_1^2$ (Proposition 5.7 together with Proposition 2.1(c) of \cite{AFJS95}), so $\mathcal{Q}^*$ is isometric to $\ell_1^2$ and hence $\mathcal{Q}$ is isometric to $\ell_\infty^2$. 
\end{rem}

\subsubsection*{Acknowledgement} 
The first-named author is supported from ANRF research grant numbered ANRF/ARG/2025/001228/MS.
The second-named author gratefully acknowledges partial financial support from the DST, Govt. of India, through the INSPIRE Faculty Fellowship (Grant No. DST/INSPIRE/04/2022/001207), and ANRF through the Prime Minister Early Career Research Grant (Grant No. ANRF/ECRG/2025/000717/PMS). The fourth-named author acknowledges support from the Prime Minister Early Career Research Grant ANRF/ECRG/2024/000699/PMS and the ANRF-MATRICS grant ANRF/ARGM/2025/000895/MTR.

\subsubsection*{Declaration of use of generative AI} Gadadhar Misra acknowledges the use of ChatGPT (OpenAI) as an interactive aid in exploring counterexamples, clarifying Paulsen’s factorisation theorem (which is not used in the proof), developing the dilation argument, and refining the exposition. The authors take full responsibility for the results and the final manuscript.

\end{document}